\documentclass[12pt, lno]{amsart}
\usepackage{amsmath}
\usepackage{amsfonts}
\usepackage{amssymb}
 \usepackage[colorlinks]{hyperref}

\usepackage[latin1]{inputenc}
\usepackage[english]{babel}
\usepackage[T1]{fontenc}

\newcommand{\R}{\mathbb{R}}
\newcommand{\Rn}{\mathbb{R}^n}

\newcommand{\W}{\mathcal{W}}

\def\diam{\qopname\relax o{diam}}

\def\dist{\qopname\relax o{dist}}

\def\b{\qopname\relax o{b}}

\newcommand{\vint}{\mathop{\hbox{\vrule height3pt depth-2.7pt width.65em}\hskip-1em \int\hskip-0.4em}\nolimits}

\usepackage{amsthm}

\swapnumbers
\theoremstyle{plain}
\newtheorem{theorem}[equation]{Theorem}
\newtheorem{lemma}[equation]{Lemma}

\newtheorem{cor}[equation]{Corollary}

\theoremstyle{definition}
\newtheorem{definition}[equation]{Definition}

\theoremstyle{remark}

\numberwithin{equation}{section}

\title{Fractional Hardy--type inequalities in domains with plump complement}
\author{David E. Edmunds}
\address{(D. E. E.) Department of Mathematics, Pevensey II Building, University of Sussex, Falmer, Brighton BN1 9QH, U.K.}
\email{davideedmunds@aol.com}

\author{Ritva Hurri-Syrj\"anen}
\address{(R. H.-S.) University of Helsinki,
Department of Mathematics and Statistics,
Gustaf H\"allstr\"omin katu 2 $\b$ ,
FI-00014 University of Helsinki, Finland.}
\email{ritva.hurri-syrjanen@helsinki.fi}

\author{Antti V. V\"ah\"akangas}
\address{(A. V. V.) University of Helsinki,
Department of Mathematics and Statistics,
Gustaf H\"allstr\"omin katu 2 $\b$ ,
FI-00014 University of Helsinki, Finland.}
\email{antti.vahakangas@helsinki.fi}

\thanks{A. V. V. was supported by the Academy of Finland, grants 75166001 and 1134757,
  and by the Finnish Academy of Science and Letters, Vilho, Yrj\"o and
  Kalle V\"ais\"al\"a Foundation}
\date{\today}

\begin{document}

\keywords{fractional Hardy-type inequality,
domain with plump complement,  
Lipschitz domain, $C^\infty$ domain}
\subjclass[2010]{46E35 (26D10)}

\begin{abstract}
We establish fractional Hardy-type inequalities in a bounded domain
with plump complement. In particular our results apply in bounded $C^{\infty}$ domains and
Lipschitz domains.
\end{abstract}

\maketitle
\markboth{\textsc{Fractional Hardy-type inequalities}}
{\textsc{D. E. Edmunds, R. Hurri-Syrj\"anen, and A. V. V\"ah\"akangas}}

\section{Introduction}
Let $\Omega$  be a proper subdomain in $\Rn$, $n\geq 2$.
Let $s\in (0,1)$ and let $p,q\in (1,\infty )$ be given such that
$0<1/p-1/q <s/n$. We investigate the inequality
\begin{equation}\label{fractional_hardy-type}
\int_{\Omega}\frac{\vert u(x)\vert^q}{\dist (x,\partial \Omega)^{q(s+n(1/q-1/p))} }\,dx
\le
c
\biggl(
\int_{\R^n}\int_{\R^n}\frac{\vert u(x)-u(y)\vert^p}{\vert x-y\vert ^{ps+n} }\,dy\,dx
\biggr)^{q/p}
\end{equation}
for every
$u\in W^{s,p}(\R^n)$ with $\mathrm{spt}\,u\subset \overline\Omega$; here the finite constant
$c$ depends only on $s,n,p,q,\Omega$.
Our work was motivated by
the following fractional order inequality
\begin{equation}\label{fractional_order_hardy}
\int_{\Omega}\frac{\vert u(x)\vert^p}{\dist (x,\partial \Omega)^{ps} }\,dx 
\le c
\int_{\Omega}\int_{\Omega}\frac{\vert u(x)-u(y)\vert^p}{\vert x-y\vert ^{ps+n} }\,dy\,dx
\end{equation}
for all $u\in C_{0}(\Omega )$ with a finite constant $c$ which depends only on $s$, $n$, $p$, and $\Omega$.
B. Dyda proved that
inequality \eqref{fractional_order_hardy} holds in $\Omega$ with $p>0$,
if one of the following conditions is valid:
\begin{enumerate}
\item if $\Omega$ is a bounded Lipschitz domain and $sp>1$,  
\item if $\Omega$ is a complement of a bounded Lipschitz domain and $sp\in (0,\infty)\backslash\{1,n\}$, 
\item if
 $\Omega$ is a complement of a point and $sp\in (0,\infty)\backslash\{n\}$, 
\item if
$\Omega$ is a domain above the graph of a Lipschitz function 
$\R^{n-1}\to\R$ and
$sp\in (0,\infty)\backslash\{1\}$,\\
\end{enumerate}
\cite[Theorem 1.1]{Dyda2004}.
He showed also that inequality \eqref{fractional_order_hardy} is false if
$\Omega$ is a bounded Lipschitz domain with $sp \le 1$ and $s < 1$. 
Inequality
\eqref{fractional_order_hardy} was proved for convex domains when
$1<p<\infty$ and $1/p<s <1$ by M. Loss and C. A. Sloane, \cite[Theorem 1.2]{LossSloane2010}. 
Inequality \eqref{fractional_order_hardy} holds in a half-space
whenever
$0<s<1$, $sp \neq 1$, $1\le p<\infty$, by R. L. Frank and R. Seiringer
\cite[Theorem 1.1]{FrankSeiringer}; 
the $p=2$-case was considered in \cite[Theorem 1.1]{BogdanDyda}. 

We prove fractional Hardy-type inequalities \eqref{fractional_hardy-type} in a bounded domain whose complement is plump in the sense of 
the following definition.
The open 
and closed $n$-dimensional Euclidean balls, centered at a point  $x$ and with radius $r>0$, are denoted by $B^n(x,r)$ and $\overline{B^n(x,r)}$,
respectively.

\begin{definition}\label{plump}
Let $n\ge 2$ and $\eta\ge 1$.
A set $A$ in $\R^n$
is $\eta$-{\em plump}  if for all $x\in \bar{A}$
and all $r\in (0,\diam(A))$ there is a point $z$ in $\overline{{B}^n(x,r)}$ with
$B^n(z,r/\eta)\subset A$. 
\end{definition}

%Examples of plump domains are Lipschitz domains, convex domains, 
%domains with a cone condition, uniform domains, and even John domains.

The following is our main theorem.

\begin{theorem}\label{fractional_hardy_plump} 
Suppose that $\Omega $ is a bounded domain in $\Rn$, $n\geq 2$, with 
an $\eta$-plump complement $\R^n\setminus \Omega$, $\eta\ge 1$.
Let $s  \in (0,1)$ and 
let $p,q\in (1,\infty)$. If
$0<1/p-1/q<s /n$,
then
\begin{align*}
&\biggl(
\int_{\Omega}\frac{\vert u(x)\vert^q}{\dist (x,\partial \Omega)^{q(s+n(1/q-1/p))} }\,dx
\biggr)^{1/q}\\
&\le
c_{s,n,p,q}\eta^{2n/q+s-n/p}
\biggl(
\int_{\R^n}\int_{\R^n}\frac{\vert u(x)-u(y)\vert^p}{\vert x-y\vert ^{sp+n} }\,dy\,dx
\biggr)^{1/p}
\end{align*}
for every
$u\in W^{s,p}(\R^n)$ with $\mathrm{spt}\,u\subset \overline\Omega$.
\end{theorem}

Examples of bounded domains with $\eta$-plump complement
include Lipschitz domains and convex domains. 
More examples are obtained by using
$K$-quasiconformal mappings $f:\R^n\to \R^n$: if $\Omega$  in $\R^n$ is 
a bounded domain
with an $\eta$-plump complement, then the image $f\Omega$ 
is also  bounded and has a $\mu$-plump complement, 
where $\mu$ depends on $n,K$ and $\eta$ only,
see e.g. \cite[Theorem 6.6]{V}.

%We give applications of this theorem, see
%Corollary \ref
%{ineq_besov}, Corollary \ref{lip_app}, and Corollary \ref{zero}.

We give applications of Theorem \ref{fractional_hardy_plump} in Section \ref{applications}.

\section{Notation and auxiliary results}

The Lebesgue measure of a measurable
set $E$ in $\R^n$ is written as $\lvert E\rvert$.
For a measurable set $E$, with a finite and positive measure, we write
\[
\vint_E f(x)\,dx=\frac{1}{\lvert E \rvert} \int_{E}f(x)\,dx\,.
\]
We write $\chi_E$ for the characteristic function of a set $E$.

Let $\Omega$ be a bounded domain in $\R^n$, $n\ge 2$, and
let $\mathcal{W}$ be its Whitney decomposition. 
For the properties of Whitney cubes $Q\in\mathcal{W}$ we refer to E. M. Stein's book,
\cite{S}. In particular, we need the inequalities
\begin{equation}\label{dist_est}
\diam(Q)\le \dist(Q,\partial \Omega)\le 4 \diam(Q)\,,\quad Q\in \mathcal{W}.
\end{equation}
We let $Q\in\mathcal{W}$ be a cube
with center $x_Q$ and side length $\ell(Q)$.  
By $tQ$, $t>0$, we mean
a cube  with sides parallel to those of $Q$ that is centered 
at $x_Q$ and whose
side length is $t\ell(Q)$.
%The Lebesgue $n$-measure of a  measurable set $E$ is denoted by $\vert E%\vert.$

We recall definition of the {\em fractional order Sobolev spaces} in a
domain $\Omega$ in $\R^n$.
For $1\le p<\infty$ and $s\in (0,1)$ we let
$W^{s,p}(\Omega)$ be the collection of all functions $f$ in
$L^p(\Omega)$ with
$||f||_{W^{s,p}(\Omega)}:=||f||_{L^p(\Omega)}+|f|_{W^{s,p}(\Omega)}<\infty$,
where
\[
|f|_{W^{s,p}(\Omega)} := \bigg(\int_\Omega \int_\Omega
\frac{|u(x)-u(y)|^p}{|x-y|^{sp+n}}\,dx\,dy\bigg)^{1/p}.
\]
The support of a function $f:\R^n\to \mathbb{C}$ is denoted by
$\mathrm{spt}\,f$, and it is the closure of the set  $\{x\,:\,f(x)\not=0\}$ in $\R^n$.
%is the Gagliardo seminorm of $f$.

The notation $a\lesssim b$ mean that an inequality $a\le cb$ holds for some constant $c>0$
whose exact value is not important. 
We use subscripts to indicate the dependence on parameters, for example,
a quantity $c_{d}$ depends on a parameter $d$.

We state  fractional Sobolev--Poincar\'e inequalities
for a cube.

\begin{lemma}\label{hardy_cube}
Let $Q$ be a cube in $\R^n$, $n\ge 2$.
Suppose that $p,q\in [1,\infty)$, and $s \in (0,1)$ satisfy
%\begin{itemize}
%\item[{\em a)}] $1<p<n/s $ and $q= np/(n-s  p)$;
$0\le 1/p-1/q<s /n$.
Then, for every $u\in L^p(Q)$,
\[
\frac{1}{|Q|}\int_{Q} |u(x)-u_Q|^q\,dx \le c|Q|^{qs  /n-q/p} \bigg(\int_{Q}\int_{Q} 
\frac{|u(x)-u(y)|^p}{|x-y|^{n+s  p}}\,dy\,dx\bigg)^{q/p}.
\]
Here the constant $c>0$ is independent of $Q$ and $u$.
\end{lemma}

\begin{proof}
The inequality follows from
\cite[Remark 4.14]{H-SV}, when $Q=[-1/2,1/2]^n.$
%\begin{align*}
%\int_{Q_0} |f(x)-f_{Q_0}|^q\,dx\le c\bigg(\int_{Q_0} \int_{Q_0}
%\frac{|f(x)-f(y)|^p}{|x-y|^{n+s  p}}\,dy\,dx\bigg)^{q/p},
%\end{align*}
%where the constant $c$ is independent of $u$ and $Q$. 
A change of variables gives the general case.
%Let us write $Q=Q(x_Q,\ell/2)$ for
%a cube in $\R^n$, centered at $x_Q$ and with
%side length $\ell$. The mapping 
%$g(x)=(x-x_Q)/\ell$, $x\in \R^n$,
%maps $Q$ onto the unit cube $Q_0=Q(0,1/2)$. Hence,
%the inverse mapping
%$g^{-1}(x)=\ell\cdot x + x_Q$ maps
%$Q_0$ onto $Q$.
%
%Let us denote $f=u\circ g^{-1}\in L^p(Q_0)$ and
%$y=g^{-1}(\omega)$. Then
%\begin{align*}
%u_Q=\vint_Q u(y)\,dy = \vint_{Q_0} u(g^{-1}(\omega))\,d\omega=f_{Q_0}.
%\end{align*}
%Using this identity and doing the same change of variables again yields
%\[
%\vint_Q |u(y)-u_Q|^q\,dy = \vint_{Q_0} |f(x) - f_{Q_0}|^q\,dx.
%\]
%On the other hand, by
%a change of variables $x=g(z)$ and $y=g(\omega)$,
%\begin{align*}
%\int_{Q_0} \int_{Q_0}
%\frac{|f(x)-f(y)|^p}{|x-y|^{n+s  p}}\,dy\,dx&=
%\int_{Q} \int_{Q}
%\frac{|f(g(z))-f(g(\omega))|^p}{|g(z)-g(\omega)|^{n+s  p}}\,\frac{d\omega}{|Q|}\,\frac{dz}{|Q|}\\
%&=\int_{Q} \int_{Q}
%\frac{|u(z)-u(\omega)|^p}{|(z-\omega)/\ell|^{n+s  p}}\,\frac{d\omega}{|Q|}\,\frac{dz}{|Q|}\\
%&=|Q|^{s  p/n-1}\int_{Q} \int_{Q}
%\frac{|u(z)-u(\omega)|^p}{|z-\omega|^{n+s  p}}\,d\omega\,dz.
%\end{align*}
\end{proof}

Let $0<\sigma <d$.
The Riesz potential of a 
function $f$  is given by
\[
I_{\sigma} f(x) =\int_{\R^d} \frac{f(y)}{|x-y|^{d-\sigma}}\,dy.
\]
The following theorem is from \cite[Theorem 1]{He}.

\begin{theorem}\label{hedberg}
Suppose that $0<\sigma<d$ and let $p,q\in (1,\infty)$. If
\[0<1/p-1/q=\sigma/d,\] then
there is a constant $c>0$ such that
inequality $||I_\sigma(f)||_q \le c||f||_{p}$ holds for every
$f\in L^p(\R^d)$.
\end{theorem}

Recall from \cite{A} that
the fractional maximal function of a locally integrable function
$f:\R^{d}\to [-\infty,\infty]$ is
\[
\mathcal{M}_\sigma f(x) = \sup_{r>0} \frac{r^{\sigma}}{|B^d(x,r)|}\int_{B^d(x,r)} |f(y)|\,dy.
\]
If $Q$ is a cube in $\R^d$ and $x\in Q$, then
\begin{equation}\label{cube_est}
 \frac{\ell(Q)^{\sigma}}{|Q|} \int_{Q} |f(y)|\,dy\le c_d\mathcal{M}_\sigma f(x).
\end{equation}
Since $0<\sigma <d$, there is a constant $c_d>0$ such that
\begin{equation}\label{mi_bound}
\mathcal{M}_\sigma f(x) \le c_d I_\sigma |f|(x)
\end{equation}
for every $x\in\R^d$.

\begin{lemma}\label{carleson_lemma}
Let $\Omega$ be a bounded domain in $\R^n$, $n\ge 2$,
and let $\mathcal{W}$ be its
Whitney decomposition.
Suppose that $1<r<p<q<\infty$ and $\kappa\ge 1$. Then
\begin{equation}\label{beta_ineq}
\begin{split}
%\begin{align}\label{beta_ineq}
&\sum_{Q\in\mathcal{W}} |\kappa Q|^{2\beta}\bigg(\vint_{\kappa Q}\vint_{\kappa Q} |g(x,y)|\,dx\,dy\bigg)^t
\\&\qquad\qquad\qquad\le c_{n,r,p,q} \kappa^{n}\bigg(\iint_{\R^n\times\R^n}  |g(x,y)|^s\,dx\,dy\bigg)^{t/s}
\end{split}
\end{equation}
%\end{align}
for every $g\in L^s(\R^n\times \R^n)$, where
 $s=p/r$, $t=q/r$ and $\beta =t/s=q/p$.
\end{lemma}

\begin{proof}
The
fractional maximal function $\mathcal{M}_\sigma$
and the Riesz potential $I_{\sigma}$ are both associated with
$\R^{d}$. Throughout this proof $d=2n$ and $\sigma = 2n(\beta-1)/t$.

Let us rewrite the left hand side of inequality \eqref{beta_ineq} as
\begin{align*}
LHS=\kappa^{n}&\sum_{Q\in\mathcal{W}} \int_{\R^n}\int_{\R^n} 
\chi_{\kappa Q}(z)\chi_{Q}(w)\\
& \bigg(\ell(\kappa Q)^{2n(\beta-1)/t} \frac{1}{|\kappa Q|}\int_{\kappa Q}\frac{1}{|\kappa Q|}\int_{\kappa Q}|g(x,y)|\,dx\,dy\bigg)^t
\,dz\,dw.
\end{align*}
By \eqref{cube_est} with $(z,w)\in \kappa Q\times Q\subset 
 \kappa Q\times \kappa Q\subset \R^d$ and  by \eqref{mi_bound}
\begin{align*}
\kappa^{-n}LHS&\lesssim \sum_{Q\in\mathcal{W}}\int_{\R^n}\int_{\R^n} \chi_{\kappa Q}(z)\chi_Q(w) \big[\mathcal{M}_\sigma g\big(z,w)\big]^t\,dz\,dw\\
&\lesssim \int_{\R^n}\int_{\R^n} \big[\mathcal{M}_\sigma g\big(z,w)\big]^t\,dz\,dw\\
&\lesssim \int_{\R^n}\int_{\R^n} \big[I_\sigma |g|\big(z,w)\big]^t\,dz\,dw.
\end{align*}
Since $1<s=p/r<t=q/r<\infty$ and
\[
\frac{r}{p}-\frac{r}{q} = \frac{\beta-1}{t}=   \frac{\sigma}{2n},
\]
we obtain $0<1/s-1/t=\sigma/2n<1$. Hence, Theorem \ref{hedberg} yields
the right hand side of inequality \eqref{beta_ineq}.
\end{proof}

\section{A proof of Theorem \ref{fractional_hardy_plump}}

We prove a fractional Hardy-type inequality
in a domain $\Omega$ whose complement
is $\eta$-plump.

\begin{proof}[Proof of Theorem \ref{fractional_hardy_plump}]
%Suppose that $u\in B^{s }_{pp}(\Omega)$. We extend $u$ into the whole space
%$\Rn$ by zero, and
%denote this extension by $u$.
%
By \cite[Theorem 3.52]{V} and inequalities \eqref{dist_est} we see that for every $Q\in\mathcal{W}$ there is a closed cube
 $Q^s$ in $\R^n$ such that
 \[Q^s\subset \R^n\setminus \overline\Omega,\quad 
\diam(Q)=\diam(Q^s),\quad \dist(Q,Q^s)\le 15\eta \diam(Q).\]
We write $Q^*:=\kappa Q$ for the dilated cube of $Q$ having the same centre as $Q$ and side length $\kappa\ell(Q)$,
$\kappa=40\eta \sqrt n$. The triangle inequality implies that
$Q^s\subset Q^*$.
Let 
\begin{equation}\label{a_def}
\alpha=s +n/q-n/p>0.
\end{equation}
Suppose that $u\in W^{s,p}(\R^n)$ has support in $\overline\Omega$.
By \eqref{dist_est},
\begin{equation*}
\int_{\Omega}
\frac{\vert u(x)\vert^q}
{\dist (x,\partial \Omega)^{\alpha q}} \,dx
\le
\sum _{Q\in \W}
\diam (Q)^{-\alpha q}
\int_{Q}\vert u(x)-u_{Q^s}\vert^q\,dx\,.
\end{equation*}
%:
For a given $Q\in\mathcal{W}$ the inclusion $Q\subset Q^*$ yields
%$\diam(Q)=\diam(Q^s)$ and $\dist(Q,Q^s)\le c\diam(Q)$,
\begin{align*}
\int_{Q}\vert u(x)-u_{Q^s}\vert^q\,dx&\lesssim \int_{Q^*} |u(x)-u_{Q^*}|^q\,dx +
|Q| |u_{Q^s}-u_{Q^*}|^q
%\\&\lesssim
%\int_{Q^*}\vert u(x)-u_{Q^*}\vert^q\,dx+\int_{Q^s} |u(x)-u_{Q^s}|^q\,dx\,.
\end{align*}
Since $|Q|=|Q^s|$ and $Q^s\subset Q^*$, we obtain
\begin{align*}
|Q| |u_{Q^s}-u_{Q^*}|^q &= \int_{Q^s} |u_{Q_s}-u_{Q^*}|^q\,dx\\
&\lesssim 
\int_{Q^s} |u(x)-u_{Q^s}|^q\,dx + \int_{Q*} |u(x)-u_{Q^*}|^q\,dx.
\end{align*}
Because $0<1/p-1/q<s /n$, there is a number
 $r\in (1,p)$ such that 
 \begin{equation}\label{mu_def}
 \mu=n(1/p-1/r)+s \in (0,s )
 \end{equation}
 and
%\begin{itemize}
%\[1<p<n/s $ and $q= np/(n-s  p).\]
$0< 1/r-1/q<\mu/n$.
%\end{itemize}
Application of Lemma \ref{hardy_cube} to the cubes $Q^*$ and $Q^s$ yields
\begin{align*}
\int_{Q} |u(x)-u_{Q^s}|^q\,dx
\lesssim|Q^*|^{1+q\mu /n-q/r} \bigg(\int_{Q^*}\int_{Q^*} 
\frac{|u(x)-u(y)|^r}{|x-y|^{n+\mu r}}\,dy\,dx\bigg)^{q/r}.
\end{align*}
Hence,
\begin{align*}
&\int_{\Omega}
\frac{\vert u(x)\vert^q}
{\dist (x,\partial \Omega)^{\alpha q}} \,dx\\
&\lesssim
\sum _{Q\in \W}
\diam (Q)^{-\alpha q}
\int_{Q}\vert u(x)-u_{Q^s}\vert^q\,dx\\
&\lesssim
\eta^{\alpha q}\sum _{Q\in \W}
|Q^*|^{1+ q(\mu /n-1/r-\alpha/n)} \bigg(\int_{Q^*}\int_{Q^*} 
\frac{|u(x)-u(y)|^r}{|x-y|^{n+\mu r}}\,dy\,dx\bigg)^{q/r}\\
%&\lesssim
%\sum _{Q\in \W}
%|Q|^{1+ q(-\alpha/n +s  /n-2/r+1/p)} \bigg(\int_{Q}\int_{Q} 
%\frac{|u(x)-u(y)|^r}{|x-y|^{nr/p+s  r}}\,dy\,dx\bigg)^{q/r}\\
&\lesssim
\eta^{\alpha q}\sum _{Q\in \W}
|Q^*|^{1+ q(\mu /n+1/r-\alpha/n)} \bigg(\vint_{Q^*}\vint_{Q^*} 
\frac{|u(x)-u(y)|^r}{|x-y|^{n+\mu r}}\,dy\,dx\bigg)^{q/r}\,.
\end{align*}
Equations \eqref{a_def} and \eqref{mu_def} imply that
\begin{equation*}
1+ q(\mu /n+1/r-\alpha/n)= 2q/p.
\end{equation*}
Hence, Lemma \ref{carleson_lemma} yields
\begin{align*}
&\int_{\Omega}
\frac{\vert u(x)\vert^q}
{\dist (x,\partial \Omega)^{\alpha q}} \,dx\\
%&\lesssim
%\sum _{Q\in \W}
%|Q|^{1+ q(-\alpha/n +\mu /n+1/r)} \bigg(\vint_{Q}\vint_{Q} 
%\frac{|u(x)-u(y)|^r}{|x-y|^{n+\mu r}}\,dy\,dx\bigg)^{q/r}\\
&\lesssim
\eta^{\alpha q}\sum _{Q\in \W}
|Q^*|^{2q/p} \bigg(\vint_{Q^*}\vint_{Q^*} 
\frac{|u(x)-u(y)|^r}{|x-y|^{n+\mu r}}\,dy\,dx\bigg)^{q/r}\\
&\lesssim
\eta^{\alpha q+n} \bigg(\iint_{\R^n\times\R^n} 
\frac{\vert u(x)-u(y)\vert^p}{\vert x-y\vert ^{ps  +n}}
\,dx\,dy\bigg)^{q/p}\,.
\end{align*}
Since $\alpha q+n = 2n+q(s-n/p)$, the claim follows.
\end{proof}

%We define
%\begin{align*}
%B^s_{pp}(\Omega)=\big\{f\in L^p(\Omega)\,:\,&f=g|\Omega\text{ for
%some } g\in W^{s,p}(\R^n)\big\},\\
%||f||_{B^{s}_{pp}(\Omega)}&=\inf ||g||_{W^{s,p}(\Omega)},
%\end{align*}
%where the infimum is taken over all $g\in W^{s,p}(\R^n)$
%with $g|\Omega=f$. Denote
%\[
%\widetilde{B}^s_{pp}(\overline\Omega)=\big\{ f\in W^{s,p}(\R^n)\,:\,\mathrm{spt}(f)\subset \overline\Omega\}.
%\]
%
%

\section{Applications of Theorem \ref{fractional_hardy_plump} }\label{applications}

Let us begin with certain function spaces.
The usual Besov space $B^{s}_{pp}(\R^n)$ 
coincides with the Sobolev space $W^{s,p}(\R^n)$, \cite[pp. 6--7]{T2}.
Hence, we may define
\begin{equation}\label{tilde_def}
\begin{split}
&\widetilde{B}^s_{pp}(\overline\Omega)=\big\{u\in W^{s,p}(\R^n)\,:\,\mathrm{spt}\,u\subset \overline\Omega\},\\&||u||_{\widetilde{B}^s_{pp}(\overline\Omega)}=||u||_{W^{s,p}(\R^n)}.
\end{split}
\end{equation}
The following corollary follows immediately from
Theorem \ref{fractional_hardy_plump}.

\begin{cor}\label{ineq_besov}
Suppose that $\Omega $ is a bounded domain in $\Rn$, $n\geq 2$, with 
an $\eta$-plump complement $\R^n\setminus \Omega$, $\eta\ge 1$.
Let $s  \in (0,1)$ and 
$p,q\in (1,\infty)$. If
$0<1/p-1/q<s /n$, then
%Let $s  \in (0,1)$ and 
%let $p,q\in (1,\infty)$ be such that
%$0<1/p-1/q<s /n$.
%%\begin{itemize}
%%\item[{\em a)}] $1<p<n/s $ and $q= np/(n-s  p)$ or
%%\item[{\em b)}] $0\le 1/p-1/q<s /n$.
%%\end{itemize}
\begin{align*}
&\biggl(
\int_{\Omega}\frac{\vert u(x)\vert^q}{\dist (x,\partial \Omega)^{sq+n(1-q/p)} }\,dx
\biggr)^{1/q}
\le
c_{s,n,p,q}\eta^{2n/q+s-n/p}
||u||_{\widetilde{B}^s_{pp}(\overline\Omega)}\end{align*}
for every $u\in \widetilde{B}^s_{pp}(\overline\Omega)$. 
\end{cor}

Related Hardy inequalities for a wider scale of
Triebel--Lizorkin and Besov spaces
$\widetilde{F}^s_{pq}(\overline\Omega)$ and
$\widetilde{B}^s_{pq}(\overline\Omega)$, respectively,
have been considered in \cite{T1}.  The novelty
in our result is that we only require the complement of $\Omega$ in $\R^n$
to be $\eta$-plump.

Let us study the
validity of an 
intrinsic Hardy-type inequality.
We focus on bounded Lipschitz domains and $C^\infty$ domains in $\R^n$,  \cite[p.64]{T3}.
In both cases,  the complement of  $\Omega$ in $\R^n$ is $\eta$-plump
for some $\eta\ge 1$.
The following corollary applies to all $u\in W^{s,p}(\Omega)$
but is restricted to the case $0<s<1/p$.

\begin{cor}\label{lip_app}
Let $\Omega$ be a bounded
Lipschitz domain in $\R^n$, $n\ge 2$.
%Under the assumptions
%of Theorem \ref{fractional_hardy_plump},
%Let $\Omega $ be a bounded domain in Euclidean $n$-space $\Rn$, $n\geq 2$, with 
%an $\eta$-plump complement.
Let $p,q\in (1,\infty)$ and $s\in (0,1/p)$. If
$0<1/p-1/q<s /n$, then
%%\begin{itemize}
%%\item[{\em a)}] $1<p<n/s $ and $q= np/(n-s  p)$ or
%%\item[{\em b)}] $0\le 1/p-1/q<s /n$.
%%\end{itemize}
%If $u\in \widetilde{B}^s_{pp}(\overline\Omega)$, then there exists a constant $c<\infty$, independent of $u$, such that
%\begin{align*}
%&\biggl(
%\int_{\Omega}\frac{\vert u(x)\vert^q}{\dist (x,\partial \Omega)^{s  q+n(1-q/p)} }\,dx
%\biggr)^{1/q}
%\\&\le
%c(s  ,n,p,q)\eta ^{?}||u||_{L^p(\Omega)}
%+ \bigg(\int_\Omega\int_\Omega \frac{|u(x)-u(y)|}{|x-y|^{sp+n}}\,dx\,dy\bigg)^{1/p}
%\end{align*}
there is a constant $c>0$ such that
the inequality 
\begin{equation}\label{hardy_lip_1}
\begin{split}
&\biggl(
\int_{\Omega}\frac{\vert u(x)\vert^q}{\dist (x,\partial \Omega)^{s  q+n(1-q/p)} }\,dx
\biggr)^{1/q}
\\&\le
c||u||_{L^p(\Omega)}
+ c\bigg(\int_\Omega\int_\Omega \frac{|u(x)-u(y)|^p}{|x-y|^{sp+n}}\,dx\,dy\bigg)^{1/p}
=c||u||_{W^{s,p}(\Omega)}
\end{split}
\end{equation}
 holds for all $u\in W^{s,p}(\Omega)$.
%Furthermore, if $0<s<1/p$, then
%\begin{align*}
%&\biggl(
%\int_{\Omega}\frac{\vert u(x)\vert^q}{\dist (x,\partial \Omega)^{s  q+n(1-q/p)} }\,dx
%\biggr)^{1/q}\\
%&\le
%c(s  ,n,p,q)\eta ^{?}||u||_{L^p(\Omega)}
%+ \bigg(\int_\Omega\int_\Omega \frac{|u(x)-u(y)|}{|x-y|^{sp+n}}\,dx\,dy\bigg)^{1/p}
%\end{align*}
%for every $u\in W^{s,p}(\Omega)$.
\end{cor}

\begin{proof}
Since $B^{s}_{pp}(\R^n)=W^{s,p}(\R^n)$
the usual  Besov space $B^s_{pp}(\Omega)$ can be defined by
\begin{align*}
B^s_{pp}(\Omega)=\big\{f\in L^p(\Omega)\,:\,&f=g|_\Omega\text{ for
some } g\in W^{s,p}(\R^n)\big\},\\
||f||_{B^{s}_{pp}(\Omega)}&=\inf ||g||_{W^{s,p}(\R^n)},
\end{align*}
where the infimum is taken over all functions $g\in W^{s,p}(\R^n)$, $g|_\Omega=f$. 
In the following two identifications we assume
that $\Omega$ is a bounded
Lipschitz domain.
First,
\[\widetilde{B}^s_{pp}(\overline\Omega)=B^{s}_{pp}(\Omega)\]
with equivalent norms, \eqref{tilde_def} and \cite[p. 66]{T3}.
The spaces $B^s_{pp}(\Omega)$ 
and $W^{s,p}(\Omega)$ coincide and the
norms are equivalent, \cite[Theorem 6.7]{ds} and \cite[Theorem 1.118]{T3}.
Inequality \eqref{hardy_lip_1} is therefore a consequence of Corollary
\ref{ineq_besov}.
\end{proof}

%As a consequence of these considerations and Theorem 
%\ref{fractional_hardy_plump} we obtain the following theorem.

The assumption $0<s<1/p$ can be relaxed if we
restrict the boundary behavior of functions. We state the following corollary.

\begin{cor}\label{zero}
Suppose that $\Omega$ is a bounded
$C^\infty$ domain in $\R^n$, $n\ge 2$.
%Under the assumptions
%of Theorem \ref{fractional_hardy_plump},
%Let $\Omega $ be a bounded domain in Euclidean $n$-space $\Rn$, $n\geq 2$, with a
%an $\eta$-plump complement.
Let $p,q\in (1,\infty)$ and $s \in (0,1)$, $s\not=1/p$. If
$0<1/p-1/q<s /n$, then
%Suppose
%also that 
%\begin{equation}\label{embedding}
%\overset{\circ}{W^{s,p}}(\Omega):=\mathrm{cl}_{W^{s,p}(\Omega)}(C^\infty_0(\Omega))
%\subset \widetilde{B}^s_{pp}(\overline\Omega),
%\end{equation}
%and the embedding is bounded. Then
%%\begin{itemize}
%%\item[{\em a)}] $1<p<n/s $ and $q= np/(n-s  p)$ or
%%\item[{\em b)}] $0\le 1/p-1/q<s /n$.
%%\end{itemize}
%If $u\in \widetilde{B}^s_{pp}(\overline\Omega)$, then there exists a constant $c<\infty$, independent of $u$, such that
%\begin{align*}
%&\biggl(
%\int_{\Omega}\frac{\vert u(x)\vert^q}{\dist (x,\partial \Omega)^{s  q+n(1-q/p)} }\,dx
%\biggr)^{1/q}
%\\&\le
%c(s  ,n,p,q)\eta ^{?}||u||_{L^p(\Omega)}
%+ \bigg(\int_\Omega\int_\Omega \frac{|u(x)-u(y)|}{|x-y|^{sp+n}}\,dx\,dy\bigg)^{1/p}
%\end{align*}
%where
there is a constant $c>0$ such that
the inequality 
\begin{equation}\label{hardy_lip_2}
\begin{split}
&\biggl(
\int_{\Omega}\frac{\vert u(x)\vert^q}{\dist (x,\partial \Omega)^{s  q+n(1-q/p)} }\,dx
\biggr)^{1/q}
\\&\le
c||u||_{L^p(\Omega)}
+ c\bigg(\int_\Omega\int_\Omega \frac{|u(x)-u(y)|^p}{|x-y|^{sp+n}}\,dx\,dy\bigg)^{1/p}
=c||u||_{W^{s,p}(\Omega)}
\end{split}
\end{equation}
 holds for all
\[u\in {W^{s,p}_0}(\Omega):=\overline{C^\infty_0(\Omega)}^{_{W^{s,p}(\Omega)}}.\]
\end{cor}

\begin{proof}
Observe that $\Omega$
is also a bounded Lipschitz domain in $\R^n$.
Hence, reasoning
as in the proof of Corollary \ref{lip_app} yields
$W^{s,p}(\Omega)=B^s_{pp}(\Omega)$ and, consequently,
\[{W^{s,p}_0}(\Omega)=
\overline{C^\infty_0(\Omega)}^{_{B^s_{pp}(\Omega)}}=
\overset{\circ}{B^s_{pp}}(\Omega).\]
Because $s\not=1/p$,
 \[\overset{\circ}{B^s_{pp}}(\Omega)= \widetilde{B}^s_{pp}(\overline\Omega);\]
we  refer to \cite[pp. 66-67]{T3}. Inequality
\eqref{hardy_lip_2} follows from these facts and
Corollary \ref{ineq_besov}.
\end{proof}

\bibliographystyle{amsalpha}

\end{document}